\documentclass[12pt,a4paper]{amsart}
\allowdisplaybreaks[1]
\usepackage{amstext}
\usepackage{amsthm}
\usepackage{amssymb}
\usepackage{mathtools}
\usepackage{color}
\usepackage[all]{xy}
\usepackage[top=30truemm,bottom=30truemm,left=25truemm,right=25truemm]{geometry}

\DeclareFontEncoding{OT2}{}{}
\DeclareFontSubstitution{OT2}{cmr}{m}{l}

\DeclareFontFamily{OT2}{cmr}{\hyphenchar\font45}
\DeclareFontShape{OT2}{cmr}{m}{l}{%
<5><6><7><8><9>gen*wncyr%
<10><10.95><12><14.4><17.28><20.74><24.88>wncyr10}{}

\DeclareMathAlphabet{\mathcyr}{OT2}{cmr}{m}{l}

\newtheorem{thm}{Theorem}[section]
\newtheorem{lem}[thm]{Lemma}
\newtheorem{prop}[thm]{Proposition}
\newtheorem{cor}[thm]{Corollary}

\theoremstyle{definition}

\theoremstyle{remark}
\newtheorem{rem}[thm]{Remark}

\newcommand{\be}{\boldsymbol{e}}
\newcommand{\bk}{\boldsymbol{k}}
\newcommand{\bl}{\boldsymbol{l}}

\newcommand{\bQ}{\mathbb{Q}}
\newcommand{\bR}{\mathbb{R}}
\newcommand{\bZ}{\mathbb{Z}}
\newcommand{\cA}{\mathcal{A}}
\newcommand{\cF}{\mathcal{F}}
\newcommand{\cI}{\mathcal{I}}
\newcommand{\cS}{\mathcal{S}}
\newcommand{\cZ}{\mathcal{Z}}
\newcommand{\frh}{\mathfrak{h}}
\newcommand{\frH}{\mathfrak{H}}
\newcommand{\frZ}{\mathfrak{Z}}
\newcommand{\ra}{\rightarrow}
\newcommand{\ua}{\uparrow}
\newcommand{\da}{\downarrow}

\DeclareMathOperator{\dep}{dep}
\DeclareMathOperator{\wt}{wt}

\begin{document}

\title[Ohno-type relation for interpolated MZV]{Ohno-type relation for interpolated multiple zeta values}

\author{Minoru Hirose}
\address[Minoru Hirose]{Institute For Advanced Research, Nagoya University, Furo-cho, Chikusa-ku, Nagoya, 464-8602, Japan}
\email{minoru.hirose@math.nagoya-u.ac.jp}

\author{Hideki Murahara}
\address[Hideki Murahara]{The University of Kitakyushu, 4-2-1 Kitagata, Kokuraminami-ku, Kitakyushu, Fukuoka, 802-8577, Japan}
\email{hmurahara@mathformula.page}

\author{Masataka Ono}
\address[Masataka Ono]{Global Education Center, Waseda University, 1-6-1, Nishi-Waseda, Shinjuku-ku, Tokyo, 169-8050, Japan}
\email{m-ono@aoni.waseda.jp}

\keywords{Multiple zeta(-star) values, Interpolated multiple zeta values, Ohno-type relations}
\subjclass[2010]{Primary 11M32; Secondary 05A19}

\begin{abstract}
We prove the Ohno-type relation for the interpolated multiple zeta values, which was introduced first by Yamamoto. Same type results for finite multiple zeta values are also given. Moreover, these relations give the sum formula for interpolated multiple zeta values and interpolated $\cF$-multiple zeta values, which were proved by Yamamoto and Seki, respectively.
\end{abstract}

\maketitle

\section{Introduction}
\subsection{Ohno-type relation for multiple zeta (star) values}
For a non-negative integer $r$ and an $r$-tuple of non-negative integers $\be=(e_1, \ldots, e_r)$, we set $\wt(\be) \coloneqq e_1+\cdots+e_r$ and $\dep(\be)\coloneqq r$, and we call them the \emph{weight} and $depth$, respectively. We call $\be$ an index if all the entries of $\be$ are positive. In particular, there exists a unique index of depth 0, which we call an \emph{empty index} and denote by $\varnothing$. We call an index $\bk=(k_1, \ldots, k_r)$ \emph{admissible} if $r\ge1$ and $k_r \ge2$, or $\bk =\varnothing$.

For an admissible index $\bk=(k_1, \ldots, k_r)$, \emph{the multiple zeta value} (MZV) $\zeta(\bk)$ and \emph{the multiple zeta-star value} (MZSV) $\zeta^{\star}(\bk)$ are real numbers defined by
\begin{align*}
\zeta(\bk)
\coloneqq
\sum_{1\le n_1 <\cdots<n_r}\frac{1}{n^{k_1}_1\cdots n^{k_r}_r},
\qquad
\zeta^{\star}(\bk)
\coloneqq
\sum_{1\le n_1 \le \cdots \le n_r}\frac{1}{n^{k_1}_1\cdots n^{k_r}_r}.
\end{align*}
We set $\zeta(\varnothing)=\zeta^{\star}(\varnothing) \coloneqq 1$. 
It is known that there exist many $\bQ$-linear relations among MZ(S)Vs. In this paper, we focus on Ohno-type relations for MZVs and MZSVs. For an admissible index $\bk=(k_1, \ldots, k_r)$, we can write
\begin{align*}
\bk=(\underbrace{1, \ldots, 1}_{a_1-1}, b_1+1, \ldots, \underbrace{1, \ldots, 1}_{a_s-1}, b_s+1)
\end{align*}
for $a_p, b_q \ge1$. Then we define \emph{the dual index} $\bk^{\dagger}$ of $\bk$ by
\begin{align*}
\bk^{\dagger}
\coloneqq
(\underbrace{1, \ldots, 1}_{b_s-1}, a_s+1, \ldots, \underbrace{1, \ldots, 1}_{b_1-1}, a_1+1).
\end{align*}
For example, we have $(2,1,3)^{\dagger}=(\{1\}^{1-1}, 1+1, \{1\}^{2-1}, 2+1)^{\dagger}=(\{1\}^{2-1}, 2+1, \{1\}^{1-1}, 1+1)=(1,3,2)$. We set $\varnothing^{\dagger} \coloneqq \varnothing$. For $\bk=(k_1, \ldots, k_r)$ and $\be=(e_1, \ldots, e_r)$, set $\bk \oplus \be \coloneqq (k_1+e_1, \ldots, k_r+e_r)$. Moreover, we set $\varnothing \oplus \varnothing \coloneqq \varnothing$.

\begin{thm}[{Ohno-type relation for MZV, Ohno \cite{Ohn99}}]\label{thm:Ohno-type_MZV}
For an admissible index $\bk$ and a non-negative integer $m$, we have
\begin{align*}
\sum_{\substack{\wt(\be)=m \\ \dep(\be)=\dep(\bk)}}
\zeta(\bk \oplus \be)
=
\sum_{\substack{\wt(\be)=m \\ \dep(\be)=\dep(\bk^{\dagger})}}
\zeta\bigl((\bk^{\dagger} \oplus \be)^{\dagger}\bigr).
\end{align*}
\end{thm}

\begin{rem}
In fact, Ohno \cite{Ohn99} proved wider class of $\bQ$-linear relations among MZVs commonly called \emph{Ohno relation}.
\end{rem}

\begin{thm}[{Ohno-type relation for MZSV, Hirose--Imatomi--Murahara--Saito \cite{HIMS20}}]\label{thm:Ohno-type_MZSV}
For an admissible index $\bk$ and a non-negative integer $m$, we have
\begin{align*}
\sum_{\substack{\wt(\be)=m \\ \dep(\be)=\dep(\bk)}}
c_1(\bk,\be)\zeta^{\star}(\bk \oplus \be)
=
\sum_{\substack{\wt(\be)=m \\ \dep(\be)=\dep(\bk^{\dagger})}}
\zeta^{\star}\bigl((\bk^{\dagger} \oplus \be)^{\dagger}\bigr).
\end{align*}
Here, $c_1(\bk, \be)=1$ if $\bk=\varnothing$, and
\begin{align*}
c_1((k_1, \ldots, k_r),(e_1, \ldots, e_r))
\coloneqq
\prod_{i=1}^r\binom{k_i+e_i+\delta_{i,1}-2}{e_i},
\quad
\binom{n-1}{n}
=
\begin{cases}
1 & \text{if $n=0$}, \\
0 & \text{otherwise},
\end{cases}
\end{align*}
and $\delta_{x,y}$ is Kronecker's delta.
\end{thm}

\subsection{Interpolated MZV}
This subsection describes the Ohno-type relation for \emph{interpolated MZVs} that interpolates both Theorems~\ref{thm:Ohno-type_MZV} and \ref{thm:Ohno-type_MZSV}. For an admissible index $\bk=(k_1, \ldots, k_r)$, Yamamoto \cite{Yam13} defined \emph{the interpolated multiple zeta value} $\zeta^{t}(\bk)$ by
\begin{equation*} \label{eq:def_tMZV}
\zeta^{t}(\bk)
\coloneqq
\sum_{\substack{\square\textrm{ is either a comma `,' } \\
 \textrm{ or a plus `+'}}}
 t^{(\text{the number of `+'})}
 \zeta(k_1 \square k_2 \square \cdots \square k_r) \in \bR[t]
\end{equation*}
and $\zeta^{t}(\varnothing)\coloneqq1$. Note that the interpolated MZV is a common generalization of MZV and MZSV since $\zeta^0(\bk)=\zeta(\bk)$ and $\zeta^1(\bk)=\zeta^{\star}(\bk)$.
Many $\bQ[t]$-linear relations among interpolated MZV $\zeta^{t}(\bk)$ are known (for example, see \cite{Li19}, \cite{LQ17}, \cite{TW16}, and \cite{Wak17}). In this paper, we prove the Ohno-type relation for interpolated MZVs. 

\subsection{Main result}
Let $\cI$ (resp.~ $\cI^t$) be the set of formal $\bQ$-linear (resp.~ $\bQ[t]$-linear) sums of indices, $R$ a $\bQ$-linear vector space, and $Z \colon \cI \rightarrow R$ a $\bQ$-linear map. If the equality
\begin{align*}
\sum_{\substack{\wt(\be)=m \\ \dep(\be)=\dep(\bk)}}
Z(\bk \oplus \be)
=
\sum_{\substack{\wt(\be)=m \\ \dep(\be)=\dep(\bk^{\dagger})}}
Z\bigl((\bk^{\dagger} \oplus \be)^{\dagger}\bigr)
\end{align*}
holds for any admissible index $\bk$ and any non-negative integer $m$, we say that \textit{$Z$ satisfies the Ohno-type relation}. For example, if $R=\bR$, then $Z=\zeta$ satisfies the Ohno-type relation by Theorem~\ref{thm:Ohno-type_MZV}.

For a non-empty index $\bk=(k_1, \ldots, k_r)$ and a non-negative integer $m$, define $g_m(\bk; t) \in \cI^t$ by
\begin{align*}
 &g_m(\bk;t) \\
 &\coloneqq \sum_{l=1}^{r}(-t(1-t))^{r-l}
  \sum_{\substack{\be=(e_1, \ldots, e_l) \in \bZ^{l}_{\ge0} \\ e_{1}+\cdots+e_{l}=m}}
  \sum_{1=i_{1}<\cdots<i_{l+1}=r+1}
  \prod_{l'=1}^{l}f_{i_{l'+1}-i_{l'}-1}(k_{i_{l'}}'+\cdots+k'_{i_{l'+1}-1},e_{l'}) \\
 &\qquad \times\left((k_{i_{1}}+\cdots+k_{i_{2}-1},\dots,k_{i_{l}}+\cdots+k_{i_{l+1}-1})\oplus\boldsymbol{e}\right),
\end{align*}
where
\begin{equation*}
 f_{i}(k,e)
 \coloneqq \sum_{j=0}^{e}{e-j \choose i}{k+e-i-2 \choose j}t^{j}(1-t)^{e-i-j} \in \bQ[t],
 \qquad
k_{j}' \coloneqq k_{j}+\delta_{j,1}.
\end{equation*}
We set $f_i(k, -1)\coloneqq0, g_{-1}(\bk; t) \coloneqq 0$, and $g_{m}(\varnothing; t)\coloneqq \delta_{m,0}\cdot\varnothing$. We define a $\bQ[t]$-linear map $I^t \colon \cI^t \rightarrow \cI^t$ by $I^{t}(\varnothing) \coloneqq \varnothing$ and
\begin{align*}
I^{t}(k_1, \ldots, k_r)
\coloneqq
\sum_{\substack{\square\textrm{ is either a comma `,' } \\
 \textrm{ or a plus `+'}}}
 t^{(\text{the number of `+'})}
(k_1 \square k_2 \square \cdots \square k_r) \in \cI^t.
\end{align*}
Note that if we extend the definition of MZV $\zeta$ and interpolated MZV $\zeta^{t}$ to the $\bQ[t]$-liner maps $\cI^t \rightarrow \bR[t]$, we have $\zeta^{t}=\zeta \circ I^t$.


\begin{thm}[{Main Theorem}] \label{thm:main}
Let $R$ be a $\bQ$-vector space and $Z \colon \cI \rightarrow R$ a $\bQ$-linear map satisfying the Ohno-type relation. We extend $Z$ to the $\bQ[t]$-linear map $Z \colon \cI^t \rightarrow R\otimes\bQ[t]$. Set $Z^t \coloneqq Z\circ I^t$. Then, for any admissible index $\bk$ and a non-negative integer $m$, we have
 \begin{align}\label{eq:main}
  Z^{t}\bigl(g_m(\bk;t)\bigr)
  =\sum_{\substack{\wt (\be)=m \\ \dep (\be)=\dep (\bk^{\dagger})}}
  Z^{t}\bigl((\bk^{\dagger}\oplus\be)^{\dagger}\bigr).
 \end{align}
\end{thm}
\begin{cor}
For any admissible index $\bk$ and a non-negative integer $m$, \[
\sum_{\substack{\wt (\be)=m \\ \dep (\be)=\dep (\bk^{\dagger})}}
Z^{t}\bigl((\bk^{\dagger}\oplus\be)^{\dagger}\bigr)
\] is a $\bQ[t]$-linear combination of $Z(\bl)$'s with $\dep(\bl)\leq \dep(\bk)$.
\end{cor}

\begin{rem}
\begin{enumerate}
\item For $\bk=(k_1, \ldots, k_r)$, we have
 \begin{align*}
  g_m(\bk;0)
  &=\sum_{\substack{e_{1}+\cdots+e_{r}=m \\ e_{1},\dots,e_{r}\ge 0 }}
   \biggl\{\prod_{i=1}^{r}f_{0}(k_{i}',e_{i})\biggr\}\times(k_{1}+e_{1},\dots,k_{r}+e_{r})\\
  &=\sum_{\substack{e_{1}+\cdots+e_{r}=m \\ e_{1},\dots,e_{r}\ge 0 }}(k_{1}+e_{1},\dots,k_{r}+e_{r})
 \end{align*}
 and
 \begin{align*}
  g_m(\bk;1)
  &=\sum_{\substack{e_{1}+\cdots+e_{r}=m \\ e_{1},\dots,e_{r}\ge 0 }}
   \prod_{i=1}^{r}
   \binom{k_{i}+e_{i}+\delta_{i,1}-2}{e_{i}}
   \times(k_{1}+e_{1},\dots,k_{r}+e_{r}).
 \end{align*}
 Therefore, if $(R, Z)=(\bR, \zeta)$, this theorem gives the Ohno-type relations for MZVs (Theorem~\ref{thm:Ohno-type_MZV}) and MZSVs (Theorem~\ref{thm:Ohno-type_MZSV}).
\item The cases $r=1,2$ are described as follows:
\begin{align}
&Z^t\bigl(g_m((k);t)\bigr) = f_0(k+1,m)Z(k+m),\label{eq:dep1}\\
&Z^t\bigl(g_m((k_1,k_2);t)\bigr)\label{eq:dep2}\\
&=-t(1-t)f_1(k_1+k_2+1, m)Z(k_1+k_2+m)\notag\\
&\quad +\sum_{\substack{e_1+e_2=m \\ e_1, e_2 \ge0}}f_0(k_1+1, e_1)f_0(k_2, e_2)
\bigl\{
Z(k_1+e_1, k_2+e_2)+Z(k_1+k_2+m)t
\bigr\}.\notag
\end{align}
In Section \ref{sec:applications}, we deduce the sum formulas for interpolated MZVs and $\cF$-MZVs from (\ref{eq:dep1}) and the special case of (\ref{eq:dep2}), respectively.
\end{enumerate}
\end{rem}
%

This paper is organized as follows. In Section 2, we give the recurrence relation of $g_m(\bk; t)$ which plays a key role for the proof of our main theorem. In Section 3, we prove our main theorem. In Section 4, as an application, we give the Ohno-type relations for interpolated $\cF$-MZVs. Moreover, we deduce the sum formula for interpolated MZV and $\cF$-MZV from each Ohno-type relation.


\section{Recurrence relation of $g_m(\bk; t)$ and $G_m(\bk; t)$}
In this section, we give the recurrence relations of $g_m(\bk; t)$ and $G_m(\bk; t) \coloneqq I^t(g_m(\bk; t))$. First, we introduce the arrow notation. For a non-empty index $\bk=(k_1, \ldots, k_r)$, set
\begin{align*}
\bk_{\ua}
\coloneqq (k_1, \ldots, k_{r-1}, k_r+1),
\quad
\bk_{\ra}
\coloneqq (k_1, \ldots, k_r, 1).
\end{align*}
Moreover, if $w =\sum_{i=1}^n a_i(t)\bk_i \in \cI^t$ and $\bk_1, \ldots, \bk_n \neq \varnothing$, we define $w_{\ua} \coloneqq \sum_{i=1}^n a_i(t)(\bk_i)_{\ua}$, $w_{\ra} \coloneqq \sum_{i=1}^n a_i(t)(\bk_i)_{\ra}$, respectively.
\begin{prop}\label{prop:g_noarrow} For a non-negative integer $m$, we have
\begin{align*}
g_m((1);t) = (1)\oplus(m).
\end{align*}
\end{prop}
\begin{proof}
It follows from the following calculation.
\begin{align*}
g_m((1);t) & = f_0(2,m)\left( (1)\oplus(m) \right) = (1)\oplus(m).\qedhere
\end{align*}
\end{proof}
\begin{prop}\label{prop:g_uparrow}
For a non-empty index $\bk$ and a positive integer $m$, we have
\begin{align*}
 g_m(\bk_{\ua};t)=g_m(\bk; t)_{\ua}+tg_{m-1}(\bk_{\ua};t)_{\ua}.
\end{align*}
\end{prop}
\begin{proof}
For $k \ge1, i\ge0$, and $e\ge1$, we have
\begin{align*}
f_i(k+1,e)
&=\sum_{j=0}^e \binom{e-j}{i} \binom{k+e-i-1}{j} t^j(1-t)^{e-i-j} \\
&=\sum_{j=0}^e \binom{e-j}{i}
\biggl\{ \binom{k+e-i-2}{j} +\binom{k+e-i-2}{j-1} \biggr\}t^j(1-t)^{e-i-j} \\
&=f_i(k,e)+tf_i(k+1,e-1).
\end{align*}
This equality also holds for $e=0$ since $f_i(k+1, 0)=f_i(k, 0)=\delta_{i,0}$ and $f_i(k+1, -1)=0$.
Thus, we obtain
\begin{align*}
&g_m(\bk_{\ua}; t)-g_m(\bk; t)_{\ua} \\
&=\sum_{l=1}^{r}(-t(1-t))^{r-l}
\sum_{\substack{e_{1}+\cdots+e_{l}=m\\ e_{1},\dots,e_{l} \ge 0}}
\sum_{1=i_{1}<\cdots<i_{l+1}=r+1}
\prod_{l'=1}^{l-1} f_{i_{l'+1}-i_{l'}-1}(k_{i_{l'}}'+\cdots+k'_{i_{l'+1}-1},e_{l'}) \\
&\qquad \times\prod_{l'=l}^{l}tf_{i_{l'+1}-i_{l'}-1}(k_{i_{l'}}'+\cdots+k'_{i_{l'+1}-1}+1,e_{l'}-1) \\
&\qquad \times\left((k_{i_{1}}+\cdots+k_{i_{2}-1},\dots,k_{i_{l}}+\cdots+k_{i_{l+1}-1})\oplus\be\right) \\
&=t\sum_{l=1}^{r}(-t(1-t))^{r-l}
\sum_{\substack{e_{1}+\cdots+e_{l}=m-1\\ e_{1},\dots,e_{l} \ge 0}}
\sum_{1=i_{1}<\cdots<i_{l+1}=r+1}
\prod_{l'=1}^{l-1} f_{i_{l'+1}-i_{l'}-1}(k_{i_{l'}}'+\cdots+k'_{i_{l'+1}-1},e_{l'}) \\
&\qquad \times\prod_{l'=l}^{l} f_{i_{l'+1}-i_{l'}-1}(k_{i_{l'}}'+\cdots+k'_{i_{l'+1}-1}+1,e_{l'}) \\
&\qquad \times\left((k_{i_{1}}+\cdots+k_{i_{2}-1},\dots,k_{i_{l}}+\cdots+k_{i_{l+1}-1})\oplus\be\right) \\
&=tg_{m-1}(\bk_{\ua}; t)_{\ua},
\end{align*}
which completes the proof.
\end{proof}
\begin{lem}\label{lem:recurrence_f}
For a positive integer $k$ and non-negative integers $e$ and $i$, we have
\begin{equation*}
f_{i+1}(k+1,e)-(1-t)f_{i+1}(k+2,e)+f_{i}(k,e)-f_{i}(k+1,e)=0.
\end{equation*}
\end{lem}
\begin{proof}
Note that
\begin{align*}
f_{i+1}(k+1,e) & =\sum_{j=0}^{e}{e-j \choose i+1}{k+e-i-2 \choose j}t^{j}(1-t)^{e-i-j-1},\\
f_{i+1}(k+2,e) & =\sum_{j=0}^{e}{e-j \choose i+1}{k+e-i-1 \choose j}t^{j}(1-t)^{e-i-j-1},\\
f_{i}(k,e) & =\sum_{j=0}^{e}{e-j \choose i}{k+e-i-2 \choose j}t^{j}(1-t)^{e-i-j},\\
f_{i}(k+1,e) & =\sum_{j=0}^{e}{e-j \choose i}{k+e-i-1 \choose j}t^{j}(1-t)^{e-i-j}.
\end{align*}
Then, we have
\begin{align*}
f_{i}(k,e) & =\sum_{j=0}^{e}{e-j \choose i}{k+e-i-2 \choose j}t^{j}(1-t)^{e-i-j}\\
 & =\sum_{j=0}^{e}{e-j+1 \choose i+1}{k+e-i-2 \choose j}t^{j}(1-t)^{e-i-j}\\
 & \quad -\sum_{j=0}^{e}{e-j \choose i+1}{k+e-i-2 \choose j}t^{j}(1-t)^{e-i-j}\\
 & =\sum_{j=0}^{e}{e-j+1 \choose i+1}{k+e-i-2 \choose j}t^{j}(1-t)^{e-i-j}
 -(1-t)f_{i+1}(k+1,e).
\end{align*}
Therefore, we obtain
\begin{align*}
(1-t)f_{i+1}(k+1, e)+f_i(k,e)
=\sum_{j=0}^{e}\binom{e-j+1}{i+1}\binom{k+e-i-2}{j}t^{j}(1-t)^{e-i-j}.
\end{align*}
Putting $k+1$ into $k$ in the above equality, we also have
\begin{multline*}
(1-t)f_{i+1}(k+2,e)+f_i(k+1,e)
=\sum_{j=0}^{e}\binom{e-j+1}{i+1}\binom{k+e-i-1}{j}t^{j}(1-t)^{e-i-j}.
\end{multline*}
Thus, we obtain
\begin{align*}
& (1-t)f_{i+1}(k+1, e)+f_i(k,e) - (1-t)f_{i+1}(k+2,e) - f_i(k+1,e)\\
& = \sum_{j=0}^{e}\binom{e-j+1}{i+1} \left( \binom{k+e-i-2}{j}  - \binom{k+e-i-1}{j} \right) t^{j}(1-t)^{e-i-j}\\
& = -\sum_{j=0}^{e}\binom{e-j+1}{i+1}  \binom{k+e-i-2}{j-1} t^{j}(1-t)^{e-i-j}\\
& \overset{j'=j-1}{=} -\sum_{j'=-1}^{e-1}\binom{e-j'}{i+1}  \binom{k+e-i-2}{j'} t^{j'+1}(1-t)^{e-i-j'-1}\\
& = -tf_{i+1}(k+1,e).
\end{align*}
This completes the proof.
\end{proof}
\begin{prop}\label{prop:g_rightarrow}
For a non-empty index $\bk$ and a non-negative integer $m$, we have
\begin{align*}
g_m(\bk_{\ra};t)=(1-t)g_m(\bk;t)_{\ua}+g_m(\bk;t)_{\ra}-(1-t)g_m(\bk_{\ua};t)+(1-t)g_{m-1}(\bk_{\ra\ua}; t).
\end{align*}
\end{prop}
\begin{proof}
Note that since
\begin{align*}
(1-t)g_{m-1}(\bk_{\ra\ua};t)_{\ua}
=\frac{1-t}{t}g_m(\bk_{\ra\ua}; t)-\frac{1-t}{t}g_m(\bk_{\ra}; t)_{\ua},
\end{align*}
the claim is equivalent to
\begin{align*}
g_m(\bk_{\ra};t)_{\ua}
&=(1-t)g_m(\bk;t)_{\ua\ua}+g_m(\bk;t)_{\ra\ua}-(1-t)g_m(\bk_{\ua};t)_{\ua}+(1-t)g_{m-1}(\bk_{\ra\ua}; t)_{\ua}\\
&=(1-t)g_m(\bk;t)_{\ua\ua}+g_m(\bk;t)_{\ra\ua}\\
&\qquad -(1-t)g_m(\bk_{\ua};t)_{\ua}+\frac{1-t}{t}g_m(\bk_{\ra\ua};t)-\frac{1-t}{t}g_m(\bk_{\ra};t)_{\ua},
\end{align*}
and this is equivalent to
\begin{equation}\label{eq:key_of_g_ra}
g_m(\bk_{\ra};t)_{\ua}=t(1-t)g_m(\bk;t)_{\ua\ua}+tg_m(\bk;t)_{\ra\ua}-t(1-t)g_m(\bk_{\ua};t)_{\ua}+(1-t)g_m(\bk_{\ra\ua};t).
\end{equation}
Since $g_m(\bk;t)$ also can be written as
\begin{align*}
g_m(\bk;t)
&=\sum_{l=1}^{{\rm dep}(\bk)}(-t(1-t))^{{\rm dep}(\bk)-l}
\sum_{\substack{e_{1}+\cdots+e_{l}=m \\ e_{1},\dots,e_{l}\geq0}}
\sum_{\substack{\bk=(\bk_{1},\dots,\bk_{l}) \\ {\rm dep}(\bk_{i})>0}
}\prod_{l'=1}^{l}f_{{\rm dep}(\bk_{l'})-1}({\rm wt}(\bk_{l'})+\delta_{l',1},e_{l'})\\
 & \ \ \times\left({\rm wt}(\bk_{1})+e_{1},\dots,{\rm wt}(\bk_{l})+e_{l}\right),
\end{align*}
we have
\begin{align} \label{eq:S_1+S_2}
g_m(\bk_{\ra};t)=S_{1}+S_{2},
\end{align}
where
\begin{align*}
S_{1} & \coloneqq\sum_{l=1}^{{\rm dep}(\bk)}(-t(1-t))^{{\rm dep}(\bk)-l}\sum_{\substack{e_{1}+\cdots+e_{l+1}=m\\
e_{1},\dots,e_{l+1}\geq0
}
}\sum_{\substack{\bk=(\bk_{1},\dots,\bk_{l})\\
{\rm dep}(\bk_{i})>0
}
}\prod_{l'=1}^{l}f_{{\rm dep}(\bk_{l'})-1}({\rm wt}(\bk_{l'})+\delta_{l',1},e_{l'})\\
 & \ \ \times f_{0}(1,e_{l+1})\times\left({\rm wt}(\bk_{1})+e_{1},\dots,{\rm wt}(\bk_{l})+e_{l},1+e_{l+1}\right)
\end{align*}
and
\begin{align*}
S_{2}
& \coloneqq\sum_{l=1}^{{\rm dep}(\bk)}(-t(1-t))^{{\rm dep}(\bk)-l+1}\\
&\ \ \times\sum_{\substack{e_{1}+\cdots+e_{l}=m\\
e_{1},\dots,e_{l}\geq0
}
}\sum_{\substack{\bk=(\bk_{1},\dots,\bk_{l})\\
{\rm dep}(\bk_{i})>0
}
}\prod_{l'=1}^{l}f_{{\rm dep}(\bk_{l'})+\delta_{l',l}-1}({\rm wt}(\bk_{l'})+\delta_{l',1}+\delta_{l',l},e_{l'})\\
 & \ \ \times\left({\rm wt}(\bk_{1})+e_{1},\dots,{\rm wt}(\bk_{l})+e_{l}\right)_{\ua}.
\end{align*}
Similarly, 
we have
\begin{align} \label{eq:S_3+S_4}
g_m(\bk_{\ra\ua};t)=S_{3}+S_{4},
\end{align}
where
\begin{align*}
S_{3}
& \coloneqq\sum_{l=1}^{{\rm dep}(\bk)}(-t(1-t))^{{\rm dep}(\bk)-l}
\sum_{\substack{e_{1}+\cdots+e_{l+1}=m \\ e_{1},\dots,e_{l+1}\geq0}}
\sum_{\substack{\bk=(\bk_{1},\dots,\bk_{l}) \\ {\rm dep}(\bk_{i})>0}}
\prod_{l'=1}^{l}f_{{\rm dep}(\bk_{l'})-1}({\rm wt}(\bk_{l'})+\delta_{l',1},e_{l'})\\
& \ \ \times f_{0}(2,e_{l+1})\times\left({\rm wt}(\bk_{1})+e_{1},\dots,{\rm wt}(\bk_{l})+e_{l},2+e_{l+1}\right)
\end{align*}
and
\begin{align*}
S_{4}
& \coloneqq\sum_{l=1}^{{\rm dep}(\bk)}(-t(1-t))^{{\rm dep}(\bk)-l+1}\\
& \ \ \times \sum_{\substack{e_{1}+\cdots+e_{l}=m \\ e_{1},\dots,e_{l}\geq0}}
\sum_{\substack{\bk=(\bk_{1},\dots,\bk_{l}) \\ {\rm dep}(\bk_{i})>0}}
\prod_{l'=1}^{l}f_{{\rm dep}(\bk_{l'})+\delta_{l',l}-1}({\rm wt}(\bk_{l'})+\delta_{l',1}+2\delta_{l',l},e_{l'})\\
& \ \ \times\left({\rm wt}(\bk_{1})+e_{1},\dots,{\rm wt}(\bk_{l})+e_{l}\right)_{\ua\ua}.
\end{align*}
Since
\begin{align*}
f_{0}(1,e) & =\sum_{j=0}^{e}{e-1 \choose j}t^{j}(1-t)^{e-j}
=\begin{cases}
1 & \text{if $e=0$,}\\
1-t & \text{if $e>0$}
\end{cases}
\end{align*}
and
\begin{equation*}
f_{0}(2,e)=\sum_{j=0}^{e}{e \choose j}t^{j}(1-t)^{e-j}=1,
\end{equation*}
we have
\begin{align} \label{eq:S_1-(1-t)S_3}
&(S_{1})_{\ua}-(1-t)S_{3}\\
&=\sum_{l=1}^{{\rm dep}(\bk)}(-t(1-t))^{{\rm dep}(\bk)-l}
\sum_{\substack{e_{1}+\cdots+e_{l+1}=m \\ e_{1},\dots,e_{l+1}\geq0}}
\sum_{\substack{\bk=(\bk_{1},\dots,\bk_{l}) \\ {\rm dep}(\bk_{i})>0}}
\prod_{l'=1}^{l}f_{{\rm dep}(\bk_{l'})-1}({\rm wt}(\bk_{l'})+\delta_{l',1},e_{l'}) \nonumber \\
&\quad \times\left(f_{0}(1,e_{l+1})-(1-t)f_{0}(2,e_{l+1})\right)
\times\left({\rm wt}(\bk_{1})+e_{1},\dots,{\rm wt}(\bk_{l})+e_{l},2+e_{l+1}\right) \nonumber \\
 &=\sum_{l=1}^{{\rm dep}(\bk)}(-t(1-t))^{{\rm dep}(\bk)-l}
 \sum_{\substack{e_{1}+\cdots+e_{l+1}=m \\ e_{1},\dots,e_{l+1}\geq0}}
 \sum_{\substack{\bk=(\bk_{1},\dots,\bk_{l}) \\ {\rm dep}(\bk_{i})>0}}
 \prod_{l'=1}^{l}f_{{\rm dep}(\bk_{l'})-1}({\rm wt}(\bk_{l'})+\delta_{l',1},e_{l'}) \nonumber \\
 &\quad  \times t\delta_{e_{l+1},0}
 \times\left({\rm wt}(\bk_{1})+e_{1},\dots,{\rm wt}(\bk_{l})+e_{l},2+e_{l+1}\right) \nonumber \\
 &=t\sum_{l=1}^{{\rm dep}(\bk)}(-t(1-t))^{{\rm dep}(\bk)-l}
 \sum_{\substack{e_{1}+\cdots+e_{l}=m \\ e_{1},\dots,e_{l}\geq0}}
 \sum_{\substack{\bk=(\bk_{1},\dots,\bk_{l}) \\ {\rm dep}(\bk_{i})>0}}
 \prod_{l'=1}^{l}f_{{\rm dep}(\bk_{l'})-1}({\rm wt}(\bk_{l'})+\delta_{l',1},e_{l'}) \nonumber \\
 & \quad \times\left({\rm wt}(\bk_{1})+e_{1},\dots,{\rm wt}(\bk_{l})+e_{l},2\right) \nonumber \\
 & =tg_m(\bk;t)_{\ra\ua}. \nonumber
\end{align}
Moreover, if we set $i \coloneqq \dep(\bk_l)-1, k \coloneqq \wt(\bk_l)+\delta_{l,1},$ and $e \coloneqq e_l$, from Lemma \ref{lem:recurrence_f}, we have
\begin{align*}
&(S_{2})_{\ua}-(1-t)S_{4}-t(1-t)g_m(\bk;t)_{\ua\ua}+t(1-t)g_m(\bk_{\ua};t)_{\ua}\\
&=\sum_{l=1}^{{\rm dep}(\bk)}(-t(1-t))^{{\rm dep}(\bk)-l+1}
\sum_{\substack{e_{1}+\cdots+e_{l}=m \\ e_{1},\dots,e_{l}\geq0}}
\sum_{\substack{\bk=(\bk_{1},\dots,\bk_{l}) \\ {\rm dep}(\bk_{i})>0}}
\prod_{l'=1}^{l-1}f_{{\rm dep}(\bk_{l'})-1}({\rm wt}(\bk_{l'})+\delta_{l',1},e_{l'})\\
 & \quad \times\bigl(f_{i+1}(k+1,e)-(1-t)f_{i+1}(k+2,e)+f_{i}(k,e)-f_{i}(k+1,e)\bigr)\\
 & \quad \times\left({\rm wt}(\bk_{1})+e_{1},\dots,{\rm wt}(\bk_{l})+e_{l}\right)_{\ua\ua}\\
 &=0.
\end{align*}
That is, we have
\begin{equation} \label{eq:S_2-(1-t)S_4}
(S_{2})_{\ua}-(1-t)S_{4}=tg_m(\bk;t)_{\ua\ua}-t(1-t)g_m(\bk_{\ua};t)_{\ua}.
\end{equation}
From \eqref{eq:S_1+S_2}, \eqref{eq:S_3+S_4}, \eqref{eq:S_1-(1-t)S_3}, and \eqref{eq:S_2-(1-t)S_4}, we completes the proof.
\end{proof}
Recall $G_m(\bk;t) \coloneqq I^t(g_m(\bk;t))$. From Propositions~\ref{prop:g_noarrow}, \ref{prop:g_uparrow} and \ref{prop:g_rightarrow}, we obtain the recurrence relations of $G_m(\bk;t)$.
\begin{prop} \label{prop:recurrence_G}
For any non-empty index $\bk$ and any non-negative integer $m$, we have
\begin{align*}
G_m((1);t)
&=(1)\oplus(m),\\
G_m(\bk_{\ua};t)
&=G_m(\bk;t)_{\ua}+tG_{m-1}(\bk_{\ua};t)_{\ua}, \\
G_m(\bk_{\ra};t)
 & =G_m(\bk;t)_{\ua}+G_m(\bk;t)_{\ra}-(1-t)G_m(\bk_{\ua};t)+(1-t)G_{m-1}(\bk_{\ra\ua};t).
\end{align*}
\end{prop}

\section{Proof of Main theorem}
In this section, we prove Theorem~\ref{thm:main}. For a non-empty index $\bk=(k_1, \ldots, k_r)$ and an integer $m\in \bZ_{\ge -1}$, set
\begin{align*}
 h_m(\bk; t)
 \coloneqq&\sum_{l=1}^{r}
 \sum_{\substack{e_{1}+\cdots+e_{l}+e_{1}'+\cdots+e_{l}'=m \\ e_1, \ldots, e_l, e'_1, \ldots, e'_l \ge0}}
 t^{r-l+e_{1}+\cdots+e_{l}}
\sum_{\substack{\bk=(\bk_{1},\dots,\bk_{l}) \\ {\rm dep}(\bk_{i})>0}}\\
 &\qquad \prod_{l'=1}^{l}{\wt(\bk_{l'})-\dep(\bk_{l'})+e_{l'}+\delta_{l',1}-2 \choose e_{l'}}\\
 &\qquad \times\left((\wt(\bk_{1}),\dots,\wt(\bk_{l}))\oplus\boldsymbol{e}\oplus\boldsymbol{e}'\right) \in \cI^t.
\end{align*}
If $m=-1$ then the above summation is empty, and therefore $h_{-1}(\bk; t)=0$.
\begin{prop} \label{prop:G=h}
For any non-empty index $\bk$ and any non-negative integer $m$, we have
\begin{align*}
 G_m(\bk;t)=h_m(\bk; t).
\end{align*}
\end{prop}
To show Proposition \ref{prop:G=h}, we will prove that $h_m(\bk;t)$ satisfies the same recurrence relations as $G_m(\bk; t)$.
\begin{prop} \label{prop:recurrence_h}
 For any non-empty index $\bk$ and any non-negative integer $m$, we have
 \begin{align*}
 h_m((1);t)
 &=(1)\oplus(m),\\
  h_m(\bk_{\ua};t)
  &=h_m(\bk;t)_{\ua}+th_{m-1}(\bk_{\ua};t)_{\ua},\\
  h_m(\bk_{\ra};t)
   & =h_m(\bk;t)_{\ua}+h_m(\bk;t)_{\ra}-(1-t)h_m(\bk_{\ua};t)+(1-t)h_{m-1}(\bk_{\ra\ua};t).
 \end{align*}
\end{prop}

We introduce the notation of the algebraic setting for interpolated MZVs \cite{Hof97} (see also \cite{Li19}, for example). Let $\bQ\langle x,y \rangle$ be the non-commutative polynomial ring over $\bQ$ with variables $x$ and $y$, and put $\frH_t \coloneqq \bQ\langle x,y \rangle[t] \supset \frH^1_t \coloneqq \bQ[t]+y\frH_t$. By corresponding $(k_1, \ldots, k_r)$ to $yx^{k_1-1}\cdots yx^{k_r-1}$, we have $\cI^t \cong \frH^1_t$ as $\bQ[t]$-modules. For an index $\bk$ and $m \in \bZ_{\ge0}$, let $\frh_m(\bk; t)$ be the element in $\frH^1_t$ corresponding to $h_m(\bk;t)$. We set $\frh_{-1}(\bk;t)\coloneqq0$.

We define the map $\sigma$ as an automorphism of $\bQ\langle x, y\rangle[[u]]$ satisfying $\sigma(x)=x$ and $\sigma(y)=y(1-xu)^{-1}$. Note that, for $k_1, \ldots, k_r \in \bZ_{\ge1}$, we have
\begin{align*}
\sigma\bigl(yx^{k_1-1}\cdots yx^{k_r-1}\bigr)
=\sum_{N=0}^{\infty}u^N\sum_{\substack{e_1+\cdots+e_r=N \\e_1, \ldots, e_r \ge0}}
yx^{k_1+e_1-1}\cdots yx^{k_r+e_r-1}.
\end{align*}
By the above notations, we can write the generating function of $\frh_m(\bk; t)$ as
\begin{align*}
\sum_{m=0}^{\infty}\frh_m(\bk; t)u^m = \sigma(X(\bk))
\end{align*}
where
\begin{align*}
X(\bk)\coloneqq \sum_{l=1}^{r}\sum_{e_{1},\dots,e_{l}\ge 0}t^{r-l}(tu)^{e_{1}+\cdots+e_{l}}
\sum_{\substack{\bk=(\bk_{1},\dots,\bk_{l}) \\ {\rm dep}(\bk_{i})>0}}
& \prod_{l'=1}^l{\wt(\bk_{l'})-\dep(\bk_{l'})+e_{l'}+\delta_{l',1}-2 \choose e_{l'}}\\
& \quad \times yx^{\wt(\bk_{1})+e_{1}-1}\cdots yx^{\wt(\bk_{l})+e_{l}-1}
\end{align*}
and $\sigma$ is naturally extended to the automorphism of $\bQ\langle x, y\rangle[[t,u]]$.
\begin{lem}\label{lem:gen_h}
For any non-empty index $\bk=(k_1, \ldots, k_r)$, we have
\begin{align*}
\sum_{m=0}^{\infty}
\frh_m(\bk;t)u^m
=y\frac{1}{1-xu}\Bigl(\frac{x}{1-xtu}\Bigr)^{k_1-1}
\prod_{l=2}^r
\left\{
\Bigl(y\frac{1-xtu}{1-xu}+xt\Bigr)\Bigl(\frac{x}{1-xtu}\Bigr)^{k_l-1}
\right\}.
\end{align*}
\end{lem}
\begin{proof}
Since
\begin{align*}
\sum_{e=0}^{\infty}{s+e \choose e}A^{e}=\frac{1}{(1-A)^{s+1}},
\end{align*}
we have
\begin{align*}
X(\bk)
&=\sum_{l=1}^r t^{r-l}\sum_{e_1, \ldots, e_l \ge0}\sum_{\substack{\bk=(\bk_{1},\dots,\bk_{l}) \\ {\rm dep}(\bk_{i})>0}}
yx^{\wt(\bk_1)-1}(xtu)^{e_1}\cdots yx^{\wt(\bk_l)-1}(xtu)^{e_l}\\
&\quad \times \prod_{l'=1}^l{\wt(\bk_l')-\dep(\bk_l')+e_{l'}+\delta_{l',1}-2 \choose e_{l'}}\\
&=\sum_{l=1}^r t^{r-l}\sum_{\substack{\bk=(\bk_{1},\dots,\bk_{l}) \\ {\rm dep}(\bk_{i})>0}}\prod_{l'=1}^l
\left\{
yx^{\wt(\bk_{l'})-1}
\Bigl(\frac{1}{1-xtu}\Bigr)^{\wt(\bk_{l'})-\dep(\bk_{l'})+\delta_{l',1}-1}
\right\}\\
&=\sum_{l=1}^r \sum_{\substack{\bk=(\bk_{1},\dots,\bk_{l}) \\ {\rm dep}(\bk_{i})>0}}
\prod_{l'=1}^l
\left\{
yx^{\dep(\bk_{l'})-\delta_{l',1}}
\Bigl(\frac{x}{1-xtu}\Bigr)^{\wt(\bk_{l'})-\dep(\bk_{l'})+\delta_{l',1}-1}t^{\dep(\bk_{l'})-1}
\right\}.
\end{align*}
Here, since
\begin{align*}
& yx^{\dep(\bk_{l'})-\delta_{l',1}}
\Bigl(\frac{x}{1-xtu}\Bigr)^{\wt(\bk_{l'})-\dep(\bk_{l'})+\delta_{l',1}-1}t^{\dep(\bk_{l'})-1}\\
& = y(1-xtu)^{1-\delta_{l',1}}\Bigl(\frac{x}{1-xtu}\Bigr)^{k_{l',1}-1} \times xt\Bigl(\frac{x}{1-xtu}\Bigr)^{k_{l',2}-1}
  \cdots \times xt\Bigl(\frac{x}{1-xtu}\Bigr)^{k_{l',\dep(\bk_{l'})}-1}
\end{align*}
for $\bk_{l'}=(k_{l',1},\dots,k_{l',\dep(\bk_{l'})})$,
we obtain
\begin{equation} \label{eq:calcX}
X(\bk)=y\Bigl(\frac{x}{1-xtu}\Bigr)^{k_1-1}
\prod_{l=2}^r
\left\{
\bigl(y(1-xtu)+xt\bigr)\Bigl(\frac{x}{1-xtu}\Bigr)^{k_l-1}
\right\}.
\end{equation}
Moreover, since
\begin{align*}
\sigma(x)=x,\qquad \sigma(y)=y\frac{1}{1-xu}, \qquad \sigma\bigl(y(1-xtu)+xt\bigr)=y\frac{1-xtu}{1-xu}+xt,
\end{align*}
we complete the proof.
\end{proof}
\begin{proof}[Proof of Proposition \ref{prop:recurrence_h}]
By definition,
\begin{align*}
h_m((1);t)=\sum_{e+e'=m}t^e{e-1 \choose e}((1)\oplus(e)\oplus(e'))=(1)\oplus(m).
\end{align*}

 From Lemma \ref{lem:gen_h},  we have
 \begin{align*}
 \sum_{m=0}^{\infty}\frh_m(\bk_{\ua};t)u^m
 &=\Bigl(\sum_{m=0}^{\infty}\frh_m(\bk; t)u^m\Bigr)\times \frac{x}{1-xtu}, \\
 \sum_{m=0}^{\infty}\frh_m(\bk; t)_{\ua}u^m
 &=\Bigl(\sum_{m=0}^{\infty}\frh_m(\bk; t)u^m\Bigr)\times x,\\
 \sum_{m=0}^{\infty}\frh_{m-1}(\bk_{\ua};t)_{\ua}u^m
 &=\Bigl(\sum_{m=0}^{\infty}\frh_m(\bk; t)u^m\Bigr)\times \frac{x^2u}{1-xtu}.
 \end{align*}
 Thus, we have
 \begin{align*}
 &\sum_{m=0}^{\infty}
 \bigl\{
 \frh_m(\bk_{\ua};t)-\frh_m(\bk; t)_{\ua}-t\frh_{m-1}(\bk_{\ua}; t)_{\ua}
 \bigr\}u^m\\
 &=\Bigl(\sum_{m=0}^{\infty}\frh_m(\bk; t)u^m\Bigr)
 \times \Bigl(\frac{x}{1-xtu}-x-t\frac{x^2u}{1-xtu}\Bigr)=0.
 \end{align*}
 Similarly, from Lemma \ref{lem:gen_h},  we have
 \begin{align*}
 \sum_{m=0}^{\infty}\frh_m(\bk_{\ra};t)u^{m}
 &=\Bigl(\sum_{m=0}^{\infty}\frh_m(\bk;t)u^{m}\Bigr)
 \times \Bigl(y\frac{1-xtu}{1-xu}+xt\Bigr),\\
 \sum_{m=0}^{\infty}\frh_m(\bk;t)_{\ua}u^{m}
 &=\Bigl(\sum_{m=0}^{\infty}\frh_m(\bk;t)u^{m}\Bigr)
 \times x,\\
 \sum_{m=0}^{\infty}\frh_m(\bk;t)_{\ra}u^{m}
 &=\Bigl(\sum_{m=0}^{\infty}\frh_m(\bk;t)u^{m}\Bigr)
 \times y,\\
 \sum_{m=0}^{\infty}\frh_m(\bk_{\ua};t)u^{m}
 &=\Bigl(\sum_{m=0}^{\infty}\frh_m(\bk;t)u^{m}\Bigr)
 \times \Bigl(\frac{x}{1-xtu}\Bigr), \\
 \sum_{m=0}^{\infty}\frh_{m-1}(\bk_{\ra\ua};t)u^{m}
 &=\Bigl(\sum_{m=0}^{\infty}\frh_m(\bk;t)u^{m}\Bigr)
 \times \Bigl(y\frac{1-xtu}{1-xu}+xt\Bigr)\Bigl(\frac{x}{1-xtu}\Bigr)u.
 \end{align*}
Therefore, we obtain
\begin{align*}
&\sum_{m=0}^{\infty}
\bigl\{\frh_m(\bk_{\ra}; t)
-\frh_m(\bk; t)_{\ua}
-\frh_m(\bk; t)_{\ra}
+(1-t)\frh_m(\bk_{\ua}; t)
-(1-t)\frh_{m-1}(\bk_{\ra\ua}; t)
\bigr\}u^m\\
&=\Bigl(\sum_{m=0}^{\infty}\frh_m(\bk;t)u^{m}\Bigr)\times S,
\end{align*}
where
\begin{align*}
S & =\left(y\frac{1-xtu}{1-xu}+xt\right)-x-y+\left(\frac{x}{1-xtu}\right)(1-t)\\
 & \quad -\left(y\frac{1-xtu}{1-xu}+xt\right)\left(\frac{x}{1-xtu}\right)(1-t)u\\
 & =0.
\end{align*}
This completes the proof.
\end{proof}
\begin{proof}[Proof of Proposition \ref{prop:G=h}]
Note that $G_m(\bk;t)$ and $h_m(\bk;t)$ are characterized by the recurrence relations given in Propositions \ref{prop:recurrence_G} and \ref{prop:recurrence_h}, respectively. Since these recurrence relations have exactly the same form, $G_m(\bk;t) = h_m(\bk;t)$.
\end{proof}
Let $\tau$ be the anti-automorphism of $\frH$ that interchanges $x$ and $y$, and extend it to $\frH[[t,u]]$ by $\tau(\sum_{m,n=0}^{\infty} c_{m,n}t^{m}u^{n}) = \sum_{m,n=0}^{\infty} \tau(c_{m,n})t^{m}u^{n}$.
\begin{prop} \label{prop:dual}
 For any non-empty index $\bk=(k_1, \ldots, k_r)$ and a non-negative integer $m$, we have
 \begin{align} \label{eq:dual}
  \sum_{m=0}^{\infty}u^{m}\sum_{\substack{\wt(\be)=m\\\dep(\be)=\dep(\bk^{\dagger})}}
  I^{t}\bigl((\bk^{\dagger}\oplus\be)^{\dagger}\bigr)
  &= \tau\sigma\tau(X(\bk)).
 \end{align}
\end{prop}
\begin{proof}
By \eqref{eq:calcX}, we have
\begin{align*}
\tau\sigma\tau(X(\bk))
=\tau\sigma\tau
\left(
y\Bigl(\frac{x}{1-xtu}\Bigr)^{k_{1}-1}
\prod_{l=2}^r
\left\{\bigl(y(1-xtu)+xt\bigr)\Bigl(\frac{x}{1-xtu}\Bigr)^{k_{l}-1}\right\}
\right).
\end{align*}
Since
\begin{align*}
\tau\sigma\tau\bigl(y(1-xtu)+xt\bigr)
 & =\tau\sigma\tau\bigl(y+(1-yu)xt\bigr)\\
 & =y+xt
\end{align*}
and
\begin{align*}
\tau\sigma\tau\left(\frac{x}{1-xtu}\right) & =\frac{1}{1-(1-yu)^{-1}xtu}(1-yu)^{-1}x\\
 & =\frac{1}{1-yu-xtu}x,
\end{align*}
we obtain
\begin{align*}
\tau\sigma\tau(X(\bk)) & =y\Bigl(\frac{1}{1-yu-xtu}x\Bigr)^{k_{1}-1}
\prod_{l=2}^r\left\{
(y+xt)\Bigl(\frac{1}{1-yu-xtu}x\Bigr)^{k_{l}-1}
\right\}\\
& =I^{t}\Bigl(
y\Bigl(\frac{1}{1-yu}x\Bigr)^{k_{1}-1}
\prod_{l=2}^r\left\{
y\Bigl(\frac{1}{1-yu}x\Bigr)^{k_{l}-1}
\right\}
\Bigr).
\end{align*}
This expression implies the equality \eqref{eq:dual}.
\end{proof}
\begin{proof}[Proof of Theorem \ref{thm:main}]
Since $g_m(\varnothing; 0)=\delta_{m,0}\cdot\varnothing$, the assertion is obvious for $\bk=\varnothing$. Assume that $\bk\neq\varnothing$. By Propositions \ref{prop:G=h} and \ref{prop:dual}, we have
\begin{align*}
& \sum_{m=0}^{\infty}u^m\left( Z^{t}\bigl(g_m(\bk; t)\bigr)-\sum_{\substack{\wt(\be)=m \\ \dep(\be)=\dep(\bk^{\dagger})}}
Z^{t}\bigl((\bk^{\dagger} \oplus \be)^{\dagger}\bigr) \right)\\
& = \hat{Z}(X(\bk)) - \hat{Z}(\tau\sigma\tau(X(\bk)))
\end{align*}
where $\hat{Z}$ is the natural extension of $Z$ to $\frH[[t,u]]$.
This expression vanishes by the assumption that $Z$ satisfies the Ohno-type relation. This completes the proof of Theorem~\ref{thm:main}.
\end{proof}
%

%


%


\section{Applications}\label{sec:applications}
\subsection{Interpolated $\cF$-multiple zeta values}
Let $\cA$ be the $\bQ$-algebra defined by
\begin{align*}
\cA
\coloneqq
\left.
\prod_{p}\bZ/p\bZ
\right/
\bigoplus_{p}\bZ/p\bZ,
\end{align*}
where $p$ runs through all the rational primes. For an index $\bk=(k_1, \ldots, k_r)$, we define $\cA$-multiple zeta value ($\cA$-MZV) $\zeta^{}_{\cA}(\bk)$ and $\cA$-multiple zeta-star value ($\cA$-MZSV) $\zeta^{\star}_{\cA}(\bk)$ as elements in $\cA$ by
\begin{align*}
\zeta^{}_{\cA}(\bk)
&\coloneqq
\Biggl(
\sum_{0<n_1<\cdots<n_r<p}\frac{1}{n^{k_1}_1\cdots n^{k_r}_r} \bmod{p}
\Biggr)_p,\\
\zeta^{\star}_{\cA}(\bk)
&\coloneqq
\Biggl(
\sum_{1\le n_1\le \cdots \le n_r\le p-1}\frac{1}{n^{k_1}_1\cdots n^{k_r}_r} \bmod{p}
\Biggr)_p.
\end{align*}
We set $\zeta^{}_{\cA}(\varnothing)=\zeta^{\star}_{\cA}(\varnothing) \coloneqq (1)_p \in \cA$.

On the other hand, let $\cZ$ denote the $\bQ$-subspace of $\bR$ generated by 1 and all MZVs $\zeta(\bk)$. For an index $\bk=(k_1, \ldots, k_r)$, we define $\cS$-multiple zeta value ($\cS$-MZV) $\zeta^{}_{\cS}(\bk)$ as an element of $\cZ/\zeta(2)\cZ$ by
\begin{align*}
\zeta^{}_{\cS}(\bk)
\coloneqq
\sum_{i=0}^{r}\zeta^{\ast}(k_1, \ldots, k_i)\zeta^{\ast}(k_r, \ldots, k_{i+1}) \bmod{\zeta(2)}\cZ.
\end{align*}
Here, $\zeta^{\ast}(\bl)$ is the constant term of $\ast$-regularized polynomial of MZV, which is an element in $\cZ$. Moreover, we define $\cS$-multiple zeta-star value ($\cS$-MZSV) $\zeta^{\star}_{\cS}(\bk)$ by
\begin{align*}
\zeta^{\star}_{\cS}(\bk)
\coloneqq
\sum_{\substack{\square\textrm{ is either a comma `,' } \\
 \textrm{ or a plus `+'}}}
 \zeta^{}_{\cS}(k_1 \square k_2 \square \cdots \square k_r) \in \cZ/\zeta(2)\cZ.
\end{align*}
We set $\zeta^{}_{\cS}(\varnothing)=\zeta^{\star}_{\cS}(\varnothing) \coloneqq1$. In \cite{KZ21}, Kaneko and Zagier conjectured that the $\cA$-MZVs and $\cS$-MZVs satisfy the same $\bQ$-linear relations.

For a non-empty index $\bk=(k_1, \ldots, k_r)=(\underbrace{1+\cdots+1}_{k_1}, \ldots, \underbrace{1+\cdots+1}_{k_r})$, we define the Hoffman dual index $\bk^{\vee}$ of $\bk$ by
\begin{align*}
\bk^{\vee}
\coloneqq
(\underbrace{1, \cdots, 1}_{k_1}+\underbrace{1, \cdots, 1}_{k_2}+\ldots+\underbrace{1, \cdots, 1}_{k_r}).
\end{align*}
For example, we have $(2,1,3)^{\vee}=(1+1, 1, 1+1+1)^{\vee}=(1,1+1+1,1,1)=(1,3,1,1)$.

\begin{thm}[{Ohno-type relation for $\cF$-MZVs, Oyama \cite{Oya18}}]\label{thm:Ohno-type_FMZV}
For a non-empty index $\bk$ and a non-negative integer $m$, we have
\begin{align*}
\sum_{\substack{\wt(\be)=m \\ \dep(\be)=\dep(\bk)}}
\zeta^{}_{\cF}(\bk \oplus \be)
=
\sum_{\substack{\wt(\be)=m \\ \dep(\be)=\dep(\bk^{\vee})}}
\zeta^{}_{\cF}\bigl((\bk^{\vee} \oplus \be)^{\vee}\bigr).
\end{align*}
\end{thm}
\begin{thm}[{Ohno-type relation for $\cF$-MZSVs, Hirose--Imatomi--Murahara--Saito \cite{HIMS20}}]\label{thm:Ohno-type_FMZSV}
For a non-empty index $\bk$ and a non-negative integer $m$, we have
\begin{align*}
\sum_{\substack{\wt(\be)=m \\ \dep(\be)=\dep(\bk)}}
c_2(\bk, \be)\zeta^{\star}_{\cF}(\bk \oplus \be)
=
\sum_{\substack{\wt(\be)=m \\ \dep(\be)=\dep(\bk^{\vee})}}
\zeta^{\star}_{\cF}\bigl((\bk^{\vee} \oplus \be)^{\vee}\bigr).
\end{align*}
Here,
\begin{align*}
c_2((k_1, \ldots, k_r),(e_1, \ldots, e_r))
\coloneqq
\prod_{i=1}^r\binom{k_i+e_i+\delta_{i,1}+\delta_{i,r}-2}{e_i},
\
\binom{n-1}{n}
=
\begin{cases}
1 & \text{if $n=0$}, \\
0 & \text{otherwise}.
\end{cases}
\end{align*}
\end{thm}
For a non-empty admissible index $\bk=(k_1, \ldots, k_r)$,
set $\bk_{\da} \coloneqq (k_1, \ldots, k_{r-1}, k_r-1)$.
Moreover, for $f = \sum_{i=1}^n a_i(t)\bk_i \in \cI^t$ where $\bk_1, \ldots, \bk_n$ is non-empty admissible indices,
set $f_{\da} \coloneqq \sum_{i=1}^n a_i(t)(\bk_i)_{\da} \in \cI^t$.
For $\cF \in \{\cA, \cS\}$ and an index $\bk=(k_1, \ldots, k_r)$, the second-named author and the third-named author studied \emph{the interpolated $\cF$-MZV}
\begin{align*} \label{eq:def_tFMZV}
\zeta^{t}_{\cF}(\bk)
\coloneqq
\zeta^{}_{\cF}\bigl(I^t(\bk)\bigr)
=\sum_{\substack{\square\textrm{ is either a comma `,' } \\
 \textrm{ or a plus `+'}}}
 t^{(\text{the number of `+'})}
 \zeta^{}_{\cF}(k_1 \square k_2 \square \cdots \square k_r)
\end{align*}
in \cite{MO19}. We set $\zeta^{t}_{\cF}(\varnothing)\coloneqq1$. For $\cF \in \{\cA, \cS\}$, we present the Ohno-type relation among interpolated $\cF$-MZVs, which is derived from Theorem~\ref{thm:main} and \ref{thm:Ohno-type_FMZV}.

\begin{thm}\label{thm:Ohno-type_tFMZV}
For a non-empty index $\bk$ and a non-negative integer $m$, we have
\begin{align*}
\zeta^{t}_{\cF}\bigl(g_m(\bk_{\ua};t)_{\da}\bigr)
=
\sum_{\substack{\wt(\be)=m \\ \dep(\be)=\dep(\bk^{\vee})}}
\zeta^{t}_{\cF}\bigl((\bk^{\vee} \oplus \be)^{\vee}\bigr).
\end{align*}
\end{thm}
\begin{rem}
Theorem~\ref{thm:Ohno-type_tFMZV} is an interpolation of Theorems~\ref{thm:Ohno-type_FMZV} and \ref{thm:Ohno-type_FMZSV}.
\end{rem}

\begin{proof}
Set $R \coloneqq \cA$ (resp.~$R\coloneqq \cZ/\zeta(2)\cZ$) for $\cF=\cA$ (resp.~$\cF=\cS$) and $Z(\bk) \coloneqq \zeta^{}_{\cF}(\bk_{\da})$.
Then, by Theorem~\ref{thm:Ohno-type_FMZV}, we have
\begin{align*}
\sum_{\substack{\wt(\be)=m \\ \dep(\be)=\dep(\bk)}}
Z(\bk \oplus \be)
&=\sum_{\substack{\wt(\be)=m \\ \dep(\be)=\dep(\bk)}}
\zeta^{}_{\cF}\bigl((\bk \oplus \be)_{\da}\bigr)\\
&=\sum_{\substack{\wt(\be)=m \\ \dep(\be)=\dep(\bk)}}
\zeta^{}_{\cF}(\bk_{\da} \oplus \be)\\
&=\sum_{\substack{\wt(\be)=m \\ \dep(\be)=\dep((\bk_{\da})^{\vee})}}
\zeta^{}_{\cF}\bigl(((\bk_{\da})^{\vee} \oplus \be)^{\vee}\bigr).
\end{align*}
Since
\begin{align*}
\sum_{\substack{\wt(\be)=m \\ \dep(\be)=\dep((\bk_{\da})^{\vee})}}
((\bk_{\da})^{\vee} \oplus \be)^{\vee}
=\sum_{\substack{\wt(\be)=m \\ \dep(\be)=\dep(\bk^{\dagger})}}
\bigl((\bk^{\dagger} \oplus \be)^{\dagger}\bigr)_{\da},
\end{align*}
we have
\begin{align*}
\sum_{\substack{\wt(\be)=m \\ \dep(\be)=\dep(\bk)}}
Z(\bk \oplus \be)
=\sum_{\substack{\wt(\be)=m \\ \dep(\be)=\dep(\bk^{\dagger})}}
\zeta^{}_{\cF}\bigl(((\bk^{\dagger} \oplus \be)^{\dagger})_{\da}\bigr)
=\sum_{\substack{\wt(\be)=m \\ \dep(\be)=\dep(\bk^{\dagger})}}
Z\bigl((\bk^{\dagger} \oplus \be)^{\dagger}\bigr),
\end{align*}
that is, we see that $Z(\bk)=\zeta^{}_{\cF}(\bk_{\da})$ satisfies the Ohno-type relation.
Therefore, by Theorem~\ref{thm:main}, we obtain
\begin{align*}
Z^{t}\bigl(g_m(\bk_{\ua};t)\bigr)
=\sum_{\substack{\wt (\be)=m \\ \dep (\be)=\dep (\bk^{\dagger})}}
Z^{t}\bigl(((\bk_{\ua})^{\dagger}\oplus\be)^{\dagger}\bigr).
\end{align*}
Since
\begin{align*}
Z^{t}\bigl(g_m(\bk_{\ua};t)\bigr)=\zeta^{t}_{\cF}\bigl(g_m(\bk_{\ua};t)_{\da}\bigr)
\end{align*}
and
\begin{align*}
\sum_{\substack{\wt (\be)=m \\ \dep (\be)=\dep (\bk^{\dagger})}}
Z^{t}\bigl(((\bk_{\ua})^{\dagger}\oplus\be)^{\dagger}\bigr)
&=\sum_{\substack{\wt (\be)=m \\ \dep (\be)=\dep (\bk^{\dagger})}}
\zeta^{t}_{\cF}\bigl((((\bk_{\ua})^{\dagger}\oplus\be)^{\dagger})_{\da}\bigr)\\
&=\sum_{\substack{\wt(\be)=m \\ \dep(\be)=\dep(\bk^{\vee})}}
\zeta^{t}_{\cF}\bigl((\bk^{\vee} \oplus \be)^{\vee}\bigr),
\end{align*}
we obtain the desired formula.
\end{proof}

\subsection{Sum formula}
In this subsection, we reprove the sum formula for interpolated MZVs which was first proved by Yamamoto as a corollary of our main theorem (Theorem~\ref{thm:main}).
\begin{thm}[{Sum formula for interpolated MZVs, Yamamoto \cite{Yam13}}]
For positive integers $k$ and $r$ with $k>r$, we have
\begin{align*}
\sum_{\substack{k_1+\cdots+k_r=k \\ k_1, \ldots, k_{r-1}\ge1, k_r\ge2}}
\zeta^{t}(k_1, \ldots, k_r)
=\Biggl\{\sum_{j=0}^{r-1}\binom{k-1}{j}t^j(1-t)^{r-1-j}\Biggr\}\zeta(k).
\end{align*}
\end{thm}
\begin{proof}
Set $(R, Z)=(\bR, \zeta)$ and $\bk=(k)$. Then we have
\begin{equation}\label{eq:RHSofSF}
Z^{t}\bigl(g_{r}(\bk;t)\bigr)
=\Biggl\{\sum_{j=0}^{r}\binom{k+r-1}{j}t^{j}(1-t)^{r-j}\Bigg\}Z(k+r).
\end{equation}
On the other hand, we have
\begin{equation}\label{eq:LHSofSF}
\begin{split}
\sum_{\substack{\wt(\be)=r \\ \dep(\be)=k-1}}
Z^{t}\bigl((\bk^{\dagger} \oplus \be)^{\dagger}\bigr)
&=\sum_{\substack{\wt(\be)=r \\ \dep(\be)=k-1}}
Z^{t}\bigl(((\underbrace{1, \ldots, 1}_{k-2}, 2)\oplus \be)^{\dagger}\bigr)\\
&=\sum_{\substack{\wt(\bl)=k+r \\ \dep(\bl)=k-1 \\ \mathrm{admissible}}}
Z^{t}\bigl( \bl^{\dagger} \bigr)\\
&=\sum_{\substack{\wt(\bl)=k+r \\ \dep(\bl)=r+1 \\ \mathrm{admissible}}}
Z^{t}\bigl( \bl \bigr).
\end{split}
\end{equation}
Therefore, from Theorem~\ref{thm:main}, \eqref{eq:RHSofSF} and \eqref{eq:LHSofSF}, by replacing $k$ with $k-r+1$ and $r$ with $r-1$, we complete the proof.
\end{proof}
%


In the final part of this section, we give another proof of Seki's result in \cite{Sek17} on the sum formula for interpolated $\cF$-MZVs as a corollary of our theorem. For a positive integer $k$, we set
\begin{align*}
\frZ_{\cF}(k)
\coloneqq
\begin{cases}
\Biggl(\displaystyle\frac{B_{p-k}}{k} \bmod{p}\Biggr)_p \in \cA &(\cF=\cA), \\
\zeta(k) \bmod{\zeta(2)\cZ} \in \cZ/\zeta(2)\cZ & (\cF=\cS).
\end{cases}
\end{align*}

\begin{thm}[{Sum formula for interpolated $\cF$-MZVs, Seki \cite{Sek17}}]
For positive integers $k$ and $r$ with $k>r$, we have
\begin{align*}
\sum_{\substack{\text{$\bk$:admissible} \\ \wt(\bk)=k, \dep(\bk)=r}}
\zeta^{t}_{\cF}(\bk)
=\left[\sum_{j=0}^{r-1}\Biggl\{\binom{k-1}{j}+(-1)^r\binom{k-1}{r-1-j}\Biggr\}t^{j}(1-t)^{r-1-j}\right]\frZ_{\cF}(k).
\end{align*}
\end{thm}

\begin{proof}
To prove this theorem, we need to calculate the difference between Theorem~\ref{thm:Ohno-type_tFMZV} with ``$\bk=\bk_1\coloneqq (k-r+1), m=r-1$" and ``$\bk=\bk_2 \coloneqq (k-r+1,1), m=r-2$" substituted.
First, we have
\begin{equation}\label{eq:RHS}
\begin{split}
&\sum_{\substack{\wt(\be)=r-1 \\ \dep(\be)=\bk^{\vee}_1}}
\zeta^{t}_{\cF}\bigr((\bk^{\vee}_1 \oplus \be)^{\vee}\bigl)
-\sum_{\substack{\wt(\be)=r-2 \\ \dep(\be)=\bk^{\vee}_2}}
\zeta^{t}_{\cF}\bigr((\bk^{\vee}_2 \oplus \be)^{\vee}\bigl)\\
&=\sum_{\substack{\wt(\be)=r-1 \\ \dep(\be)=k-r+1}}
\zeta^{t}_{\cF}\bigr(((\{1\}^{k-r+1}) \oplus \be)^{\vee}\bigl)
-\sum_{\substack{\wt(\be)=r-2 \\ \dep(\be)=k-r+1}}
\zeta^{t}_{\cF}\bigr(((\{1\}^{k-r}, 2) \oplus \be)^{\vee}\bigl)\\
&=\sum_{\substack{l_1+\cdots+l_{k-r}=k-1 \\ l_1, \ldots, l_{k-r}\ge1}}
\zeta^{t}_{\cF}\bigl((l_1, \ldots, l_{k-r}, 1)^{\vee}\bigr)
=\sum_{\substack{k_1+\cdots+k_r=k \\ k_1, \ldots, k_{r-1}\ge1, k_r\ge2}}
\zeta^{t}_{\cF}(k_1, \ldots, k_r).
\end{split}
\end{equation}
On the other hand, by the definition of $g_m(\bk;t)$, $\zeta^{}_{\cF}(a, b)=(-1)^{b}\binom{a+b}{a}\frZ_{\cF}(a+b) \; (a, b \in \bZ_{\ge1})$, and $\zeta^{}_{\cF}(a)=0 \; (a \in \bZ_{\ge1})$ (see \cite{Kan19}, for example), we have
\begin{equation}\label{eq:LHS}
\begin{split}
&\zeta^{t}_{\cF}\bigl(g_{r-1}((\bk_1)_{\ua};t)_{\da}\bigr)-
\zeta^{t}_{\cF}\bigl(g_{r-2}((\bk_2)_{\ua}; t)_{\da}\bigr)\\
&=-\sum_{\substack{e_1+e_2=r-2 \\ e_1, e_2 \ge0}}
\Biggl\{\sum_{j_1=0}^{e_1}\binom{k-r+e_1}{j_1}t^{j_1}(1-t)^{e_1-j_1}\Biggr\}\\
&\quad\times\Biggl\{\sum_{j_2=0}^{e_2}\binom{e_2}{j_2}t^{j_2}(1-t)^{e_2-j_2}\Biggr\}\zeta^{}_{\cF}(k-r+1+e_1, 1+e_2)\\
&=\sum_{e=0}^{r-2}(-1)^{r-e}
\Biggl\{\sum_{j=0}^{e}\binom{k-r+e}{j}t^{j}(1-t)^{e-j}\Biggr\}
\binom{k}{r-e-1}\frZ_{\cF}(k).
\end{split}
\end{equation}
Therefore, from Theorem~\ref{thm:Ohno-type_tFMZV}, \eqref{eq:RHS}, and \eqref{eq:LHS}, it suffices to prove that
\begin{equation}\label{eq:claim}
\begin{split}
&\sum_{e=0}^{r-2}(-1)^{r-e}\Biggl\{\sum_{j=0}^{e}\binom{k-r+e}{j}t^{j}(1-t)^{e-j}\Biggr\}\binom{k}{r-e-1}\\
&=\sum_{j=0}^{r-1}\Biggl\{\binom{k-1}{j}+(-1)^r\binom{k-1}{r-1-j}\Biggr\}t^{j}(1-t)^{r-1-j}.
\end{split}
\end{equation}
This equality can be proved as follows. Since $1=\sum_{s=0}^{r-e-1}\binom{r-e-1}{s}t^{s}(1-t)^{r-e-1-s}$, we have
\begin{align*}
&\sum_{e=0}^{r-2}(-1)^{r-e}\Biggl\{\sum_{j=0}^{e}\binom{k-r+e}{j}t^{j}(1-t)^{e-j}\Biggr\}\binom{k}{r-e-1}\\
&=\sum_{e=0}^{r-2}\sum_{j=0}^{e}\sum_{s=0}^{r-1-e}(-1)^{r-e}\binom{k-r+e}{j}\binom{k}{r-e-1}\binom{r-e-1}{s}t^{j+s}(1-t)^{r-1-j-s}\\
&=\sum_{\substack{j+s\le r-1 \\ j, s \ge0}}\sum_{e=j}^{\min(r-s-1, r-2)}(-1)^{r-e}
\frac{k!}{j!s!(k-r+e-j)!(r-e-1-s)!}\frac{t^{j+s}(1-t)^{r-1-j-s}}{k-r+e+1}\\
&=A+\sum_{j=0}^{r-1}\binom{k-1}{j}t^j(1-t)^{r-1-j},
\end{align*}
where, we set
\begin{align*}
A
\coloneqq
\sum_{\substack{j+s\le r-1 \\ j, s \ge0}}\sum_{e=j}^{r-s-1}(-1)^{r-e}
\frac{k!}{j!s!(k-r+e-j)!(r-e-1-s)!}\frac{t^{j+s}(1-t)^{r-1-j-s}}{k-r+e+1}.
\end{align*}
By setting $u \coloneqq j+s$ and $m \coloneqq e-j$, we have
\begin{align*}
A
&=\sum_{u=0}^{r-1}t^{u}(1-t)^{r-1-u}\sum_{j=0}^{u}\sum_{m=0}^{r-1-u}(-1)^{r-j-m}\\
&\quad \times\frac{k!}{j!(u-j)!(k-r+m)!(r-1-u-m)!}\frac{1}{k-r+m+j+1}\\
&=\sum_{u=0}^{r-1}t^u(1-t)^{r-1-u}\sum_{m=0}^{r-1-u}(-1)^{r-m}\frac{k!}{(k-r+m)!(r-1-u-m)!u!}\\
&\quad \times \sum_{j=0}^u(-1)^{j}\binom{u}{j}\frac{1}{k-r+m+1+j}.
\end{align*}
Moreover, since
\begin{align*}
\sum_{j=0}^u(-1)^{j}\binom{u}{j}\frac{1}{k-r+m+1+j}
&=\int^{1}_{0}\sum_{j=0}^u(-1)^{j}\binom{u}{j}z^{k-r+m+j}dz\\
&=\int^{1}_{0}z^{k-r+m}(1-z)^{u}dz=\frac{u!(k-r+m)!}{(u+k-r+m+1)!},
\end{align*}
we have
\begin{align*}
A
&=\sum_{u=0}^{r-1}t^u(1-t)^{r-1-u}\sum_{m=0}^{r-1-u}(-1)^{r-m}\frac{k!}{(r-1-u-m)!(u+k-r+m+1)!}\\
&=\sum_{u=0}^{r-1}t^u(1-t)^{r-1-u}\sum_{n=0}^{r-1-u}(-1)^{u+1+n}\binom{k}{n}\\
&=\sum_{u=0}^{r-1}t^u(1-t)^{r-1-u}(-1)^{r}\binom{k-1}{r-1-u}.
\end{align*}
Thus we obtain \eqref{eq:claim}, which completes the proof.
\end{proof}

\section*{Acknowledgement}
This research was supported in part by JSPS KAKENHI Grant Numbers JP16H06336, JP18J00982 and JP18K13392.


\begin{thebibliography}{99}


\bibitem{HIMS20}
M.~Hirose, K.~Imatomi, H.~Murahara, and S.~Saito,
\textit{Ohno type relations for classical and finite multiple zeta-star values},
to appear in Kyushu J.~Math.

\bibitem{Hof97}
M. E. Hoffman,
\textit{The algebra of multiple harmonic series},
J.~Algebra \textbf{194} (1997), 477--495.






\bibitem{Kan19}
M. Kaneko,
\textit{An introduction to classical and finite multiple zeta values},
Publications math\'{e}matiques de {B}esan\c{c}on, no.\ 1 (2019), 103--129.


\bibitem{KZ21}
M. Kaneko and D. Zagier,
\textit{Finite multiple zeta values},
in preparation.




\bibitem{Li19}
Z-h.\ Li,
\textit{Algebraic relations of interpolated multiple zeta values},
preprint, arXiv:1904.09887.

\bibitem{LQ17}
Z-h.\ Li and C. Qin,
\textit{Some relations of interpolated multiple zeta values},
Int.\ J. Math.\ \textbf{28} (2017), art.\ 175033 (25 pp).

\bibitem{MO19}
H. Murahara and M. Ono,
\textit{Yamamoto's interpolation of finite multiple zeta and zeta-star values},
to appear in Tokyo J.~Math.

\bibitem{Ohn99}
Y.~Ohno,
\textit{A generalization of the duality and sum formulas on the multiple zeta values},
J.~Number Theory \textbf{74} (1999), 39--43.



\bibitem{Oya18}
K.~Oyama,
\textit{Ohno-type relation for finite multiple zeta values},
Kyushu J.~Math.\ \textbf{72} (2018), 277--285.



\bibitem{Sek17}
S. Seki,
\textit{Finite multiple polylogarithms},
Doctoral Thesis (Osaka university knowledge archive).



\bibitem{TW16}
T.~Tanaka and N.~Wakabayashi,
\textit{Kawashima's relations for interpolated multiple zeta values},
J.~Algebra \textbf{447} (2016), 424--431.

\bibitem{Yam13}
S. Yamamoto,
\textit{Interpolation of multiple zeta and zeta-star values},
J.~Algebra \textbf{385} (2013), 102--114.


\bibitem{Wak17}
N. Wakabayashi,
\textit{Double shuffle and Hoffman's relations for interpolated multiple zeta values},
Int.\ J. Number Theory \textbf{13} (2017), 2245--2251.


\end{thebibliography}
\end{document}